\documentclass[11pt,a4paper]{article}
\usepackage[utf8]{inputenc}
\usepackage{amsmath,amssymb,amsthm}
\usepackage[a4paper,margin=27mm]{geometry}
\usepackage[expansion=false]{microtype}
\usepackage[hidelinks]{hyperref}
\DeclareTextSymbolDefault{\textdagger}{OMS}
\newcommand{\Prob}{\mathbb P}
\newcommand{\E}{\mathbb E}
\newcommand{\ind}{\mathbf 1}
\DeclareMathOperator{\erf}{erf}
\DeclareMathOperator{\erfc}{erfc}
\newtheorem{proposition}{Proposition}
\hypersetup{
  pdftitle={A clarification of Fermi's formula for the maximum of partial sums},
  pdfsubject={Survival probability and the diffusion limit},
  pdfauthor={Alessandro Lodi}
}
\title{\bfseries A clarification of Fermi's formula\\
for the maximum of partial sums}
\author{Alessandro Lodi\thanks{Independent Researcher.
  Email: \href{mailto:lodialessandro0@gmail.com}{\texttt{lodialessandro0@gmail.com}}.}}
\date{\small 11 October 2026}

\begin{document}
\setcounter{footnote}{1} 
\maketitle
\vspace{-1.2em}
\begin{abstract}
In Fermi's 1926 note, the expression assigned to the probability that
all partial sums remain below a threshold is the complementary
crossing probability. We obtain the correct probability from the same
diffusion equation and locate the inconsistency in the final
integration. We then state the limit for independent increments under
the Lindeberg condition. A counterexample with smooth densities shows
that small individual variances do not suffice to justify this limit.
The probabilistic results are classical; the purpose of this note is
to reconcile the stated event with the printed calculation and to
make the sufficient hypotheses for the discrete limit explicit.
\end{abstract}

\section{The event and the printed formula}

Let \(Y_1,\ldots,Y_n\) be independent random variables with mean zero
and finite variances \(\sigma_i^2\). Write
\begin{equation}\label{eq:defs}
 S_0=0,\qquad S_k=\sum_{i=1}^k Y_i,\qquad
 M_n=\max_{0\le k\le n}S_k,\qquad V_n=\sum_{i=1}^n\sigma_i^2.
\end{equation}
For a threshold \(a>0\), the problem stated by Fermi \cite{fermi},
p.~314, formula (8), is to determine
\begin{equation}\label{eq:event}
 Q_n(a)=\Prob(S_1\le a,\ldots,S_n\le a)=\Prob(M_n\le a).
\end{equation}
In the original notation, \(\sigma_i^2=K_i^2\). For \(V_n>0\),
formula (9), repeated on p.~318, assigns to this event the expression
\begin{equation}\label{eq:printed}
 \frac{2}{\sqrt{\pi}}\int_{a/\sqrt{2V_n}}^\infty e^{-z^2}\,dz
 =\erfc\!\left(\frac{a}{\sqrt{2V_n}}\right)
 =2\left[1-\Phi\!\left(\frac{a}{\sqrt{V_n}}\right)\right],
\end{equation}
where \(\Phi\) denotes the standard normal distribution function.

The event itself provides a first check. \(Q_n(a)\) is nondecreasing
in \(a\), whereas \eqref{eq:printed} is decreasing. At fixed \(n\),
\(Q_n(a)\to1\) as \(a\to\infty\), whereas \eqref{eq:printed} tends
to zero. These facts do not depend on the diffusion approximation.
The printed expression has the behavior of the complementary event.

The correction follows from Fermi's construction. The minus sign in
his formula (19) is already correct. The printed pages do not
establish whether the inconsistency arose in composition or in the
calculation. A related modern study of Fermi's note is due to
Flandoli \cite{flandoli}; no claim of priority is made here.

\section{The density with absorption}

Let \(B_t\) be standard Brownian motion started at zero, and set
\[
 \tau_a=\inf\{t\ge0:B_t=a\},\qquad
 g_t(x)=\frac{1}{\sqrt{2\pi t}}\exp\!\left(-\frac{x^2}{2t}\right).
\]
The time \(t\) represents accumulated variance, as in Fermi's
substitution (17). The measure
\[
 \Prob(B_t\in dx,\ \tau_a>t)=q_a(t,x)\,dx,\qquad x<a,
\]
has mass less than one: it retains only paths that have not reached
the threshold.

Reflection after the first arrival at \(a\) preserves Brownian law
\cite{lalley}. For \(x<a\), it identifies the density of paths
that have reached the threshold and ended at \(x\) with the free
density at \(2a-x\). Hence
\begin{equation}\label{eq:images}
 q_a(t,x)=g_t(x)-g_t(2a-x),\qquad t>0,\quad x<a.
\end{equation}
This density is nonnegative, since
\((2a-x)^2-x^2=4a(a-x)>0\). It also satisfies
\begin{equation}\label{eq:heat}
 \partial_t q_a=\tfrac12\partial_{xx}q_a,\qquad
 q_a(t,a)=0,\qquad q_a(t,x)\longrightarrow0\quad(x\to-\infty),
\end{equation}
with initial datum \(\delta_0\), interpreted as weak convergence
of measures. These are the equation and boundary conditions in
Fermi's formulas (14)--(18), in modern notation.

Integrating \eqref{eq:images} gives
\begin{align}
 Q(a,t)
 &:=\Prob\!\left(\sup_{0\le s\le t}B_s\le a\right)
   =\int_{-\infty}^a q_a(t,x)\,dx \notag\\
 &=\Phi(a/\sqrt t)-[1-\Phi(a/\sqrt t)] \notag\\
 &=2\Phi(a/\sqrt t)-1
   =\erf\!\left(\frac{a}{\sqrt{2t}}\right). \label{eq:correct}
\end{align}
The distinction between \(\tau_a>t\) and a maximum at most \(a\)
does not affect this probability: the Brownian maximum has no atom
at \(a>0\). The crossing probability is therefore
\begin{equation}\label{eq:crossing}
 H(a,t)=1-Q(a,t)
 =\erfc\!\left(\frac{a}{\sqrt{2t}}\right).
\end{equation}

\section{Where the integration changes the event}

Let \(u_a(s)\) denote the first-passage density:
\begin{equation}\label{eq:hit}
 u_a(s)=\frac{a}{\sqrt{2\pi s^3}}
          \exp\!\left(-\frac{a^2}{2s}\right),\qquad s>0.
\end{equation}
Fermi's formula (19) then becomes
\begin{equation}\label{eq:convolution}
 q_a(t,x)=g_t(x)-\int_0^t u_a(s)\,g_{t-s}(x-a)\,ds,
 \qquad x<a.
\end{equation}
The second term removes paths that have already reached the
threshold. This decomposition follows from the strong Markov
property at \(\tau_a\).

For every \(s<t\), Gaussian symmetry gives
\(\int_{-\infty}^a g_{t-s}(x-a)\,dx=1/2\).
Integrating \eqref{eq:convolution}, with Tonelli's theorem applied
to its nonnegative integral term, yields
\begin{equation}\label{eq:integrate}
 \int_{-\infty}^a q_a(t,x)\,dx
 =\Phi(a/\sqrt t)-\frac12\int_0^t u_a(s)\,ds.
\end{equation}
The substitution \(z=a/\sqrt{2s}\) gives
\begin{equation}\label{eq:uint}
 \int_0^t u_a(s)\,ds
 =\frac{2}{\sqrt{\pi}}\int_{a/\sqrt{2t}}^\infty e^{-z^2}\,dz
 =2[1-\Phi(a/\sqrt t)].
\end{equation}
Substituting \eqref{eq:uint} into \eqref{eq:integrate} recovers
\eqref{eq:correct}. On p.~318, immediately after formula (21),
Fermi's printed calculation instead identifies the integral of the
surviving density with \eqref{eq:uint}. Formula (19) and that
printed integration cannot both be correct.

The same distinction appears in the mass balance:
\begin{equation}\label{eq:flux}
 \frac{d}{dt}Q(a,t)
 =\frac12\,\partial_xq_a(t,a)=-u_a(t),\qquad Q(a,0)=1.
\end{equation}
The remaining mass is one minus the absorbed mass.

\section{A precise limit for discrete sums}

The Brownian formula is exact for the continuous process. For
discrete sums, the limiting procedure must be specified.
For each \(n\), consider independent increments
\(Y_{n,1},\ldots,Y_{n,n}\) with mean zero, variances
\(\sigma_{n,k}^2\), and total variance
\(V_n=\sum_{k=1}^n\sigma_{n,k}^2>0\). Define
\(S_{n,k}\) and \(M_n\) as in \eqref{eq:defs}.

\begin{proposition}\label{prop:limit}
Suppose that, for every \(\varepsilon>0\), the Lindeberg condition holds:
\begin{equation}\label{eq:lindeberg}
 L_n(\varepsilon)
 :=\frac{1}{V_n}\sum_{k=1}^n
 \E\!\left[Y_{n,k}^2
 \ind_{\{|Y_{n,k}|>\varepsilon\sqrt{V_n}\}}\right]
 \longrightarrow0.
\end{equation}
If \(a_n\ge0\) and \(a_n/\sqrt{V_n}\to b\in[0,\infty)\), then
\begin{equation}\label{eq:limit}
 \Prob(M_n\le a_n)\longrightarrow 2\Phi(b)-1.
\end{equation}
Consequently,
\begin{equation}\label{eq:approx}
 \Prob(M_n\le a_n)
 =\erf\!\left(\frac{a_n}{\sqrt{2V_n}}\right)+o(1).
\end{equation}
\end{proposition}

\begin{proof}
Parametrize the sums by accumulated variance:
\[
 t_{n,k}=\frac{1}{V_n}\sum_{j=1}^k\sigma_{n,j}^2,
 \qquad t_{n,0}=0,\qquad t_{n,n}=1.
\]
Increments with zero variance may be removed. For every
\(\varepsilon>0\),
\[
 \max_k\frac{\sigma_{n,k}^2}{V_n}
 \le\varepsilon^2+L_n(\varepsilon).
\]
Thus the largest gap between consecutive times tends to zero.
The functional Lindeberg theorem, on this deterministic variance
clock, gives convergence to \(B\) of the process that linearly
interpolates the points
\((t_{n,k},S_{n,k}/\sqrt{V_n})\); see \cite{durrett}, Section~8.2,
for the martingale formulation. Independence makes the conditional
variances deterministic and reduces the tail condition to
\eqref{eq:lindeberg}.

The maximum of the interpolation is \(M_n/\sqrt{V_n}\).
Since the maximum is a continuous functional on continuous paths
with the uniform norm, it converges in distribution to
\(\sup_{0\le t\le1}B_t\). The distribution function of this maximum
is continuous, including at zero, and is given by \eqref{eq:correct}
at \(t=1\), extended to \(a=0\). This proves \eqref{eq:limit}.
Continuity of \(\Phi\) then gives \eqref{eq:approx}.
\end{proof}

This result controls absolute error. When
\(a_n/\sqrt{V_n}\to0\), it does not by itself provide a small
relative error or an asymptotic equivalent for the discrete
probability.

For independent and identically distributed increments with mean
zero and variance \(\sigma^2\in(0,\infty)\), the finite second
moment implies the Lindeberg condition. At a fixed threshold \(a>0\),
\eqref{eq:limit} gives
\begin{equation}\label{eq:fixed}
 \lim_{n\to\infty}\Prob(M_n\le a)=0,\qquad
 \lim_{n\to\infty}\Prob(M_n>a)=1.
\end{equation}
The limit one at the end of Fermi's note therefore belongs to the
crossing probability.

\section{Why variance comparisons alone do not suffice}

Fermi assumes that no individual variance is appreciable compared
with their sum, and that the threshold is large compared with the
standard deviations of the individual increments. His footnote on
p.~315 also assumes the conditions needed to justify the diffusion
approximation, without stating them. Condition \eqref{eq:lindeberg}
makes one sufficient requirement explicit. The following example
shows why variance and threshold comparisons alone cannot replace it.

Let \(J_{n,k}\) be independent, with
\[
 \Prob(J_{n,k}=1)=\Prob(J_{n,k}=-1)=\frac{1}{2n},
 \qquad \Prob(J_{n,k}=0)=1-\frac1n,
\]
and let \(Z_{n,k}\) be independent standard normal variables,
independent also of the \(J_{n,k}\). Set
\[
 Y_{n,k}=J_{n,k}+\frac{Z_{n,k}}{n}.
\]
Each increment has a smooth density and mean zero. Moreover,
\begin{equation}\label{eq:examplevar}
 \sigma_{n,k}^2=\frac1n+\frac1{n^2},\qquad
 V_n=1+\frac1n,\qquad
 \frac{\max_k\sigma_{n,k}^2}{V_n}=\frac1n.
\end{equation}
For \(a=1/4\), we also have
\(a/\max_k\sigma_{n,k}\to\infty\).
Yet with probability \((1-1/n)^n\to e^{-1}\), all the
\(J_{n,k}\) vanish. On this event only the Gaussian noise remains,
with total variance \(1/n\). Kolmogorov's maximal inequality gives
\[
 \Prob\!\left(\max_{k\le n}
 \left|\sum_{j=1}^k Z_{n,j}/n\right|>a\right)
 \le\frac{1}{na^2}\longrightarrow0.
\]
By independence,
\begin{equation}\label{eq:counterexample}
 \liminf_{n\to\infty}\Prob(M_n\le1/4)
 \ge e^{-1}=0.367879\ldots
 >2\Phi(1/4)-1=0.197413\ldots.
\end{equation}
The Brownian limit fails. Rare jumps of unit size survive
normalization, and the Lindeberg condition is not satisfied.

\section{Discrete and continuous absorption}

For increments with densities \(f_k\), the discrete density of paths
still below the threshold obeys the exact recursion
\begin{equation}\label{eq:recursion}
 q_k(x)=\ind_{\{x\le a\}}\int_{-\infty}^a
                 q_{k-1}(y)f_k(x-y)\,dy,\qquad q_0=\delta_0.
\end{equation}
At the first step this means
\(q_1(x)=\ind_{\{x\le a\}}f_1(x)\), with the initial integral
understood against the measure \(\delta_0\).
The required probability is \(\int_{-\infty}^a q_n(x)\,dx\).
This recursion retains the distribution of the increments.
At a finite number of steps the density approaching the threshold
from below may be positive. The zero boundary value of the
continuous model is not an identity for the discrete recursion.

Under \eqref{eq:lindeberg}, on the variance clock, the discrete
process converges to Brownian motion and yields \eqref{eq:approx}.
Fermi's formulas (9) and (9) bis must therefore be complemented
when assigned to the event (8). The diffusion method retains its
meaning once surviving mass, absorbed mass, and the conditions for
the limit are distinguished.

\section*{Use of generative AI}
An OpenAI language model assisted with the mathematical reconstruction,
drafting, English revision, source checking, and preparation of the
numerical verification script. This assistance extended beyond
language editing. Responsibility for the scientific contents rests
with the human author.

\end{document}